\documentclass[12pt,a4paper]{amsart}
\usepackage{amssymb,amscd,hyperref}

\newtheorem{theorem}{Theorem}[section]
\newtheorem{definition}[theorem]{Definition}
\newtheorem{proposition}[theorem]{Proposition}
\newtheorem{lemma}[theorem]{Lemma}

\begin{document}

\title{A maximality result for orthogonal quantum groups}

\author{Teodor Banica}
\address{T.B.: Department of Mathematics, Cergy-Pontoise University, 95000 Cergy-Pontoise, France. {\tt teodor.banica@u-cergy.fr}}

\author{Julien Bichon}
\address{J.B.: Department of Mathematics, Clermont-Ferrand University, Campus des Cezeaux, 63177 Aubiere Cedex, France. {\tt bichon@math.univ-bpclermont.fr}}

\author{Beno\^ it Collins}
\address{B.C.: Department of Mathematics, Lyon 1 University, and University of Ottawa, 585 King Edward, Ottawa, ON K1N 6N5, Canada. {\tt bcollins@uottawa.ca}}

\author{Stephen Curran}
\address{S.C.: Department of Mathematics, University of California, Los Angeles, CA 90095, USA. {\tt curransr@math.ucla.edu}}


\begin{abstract}
We prove that the quantum group inclusion $O_n \subset O_n^*$ is ``maximal'', where $O_n$ is the usual orthogonal group and $O_n^*$ is the half-liberated orthogonal quantum group, in the sense that there is no intermediate compact quantum group $O_n\subset G\subset O_n^*$. In order to prove this result, we use: (1) the isomorphism of projective versions $PO_n^*\simeq PU_n$, (2) some maximality results for classical groups, obtained by using Lie algebras and some matrix tricks, and (3) a short five lemma for cosemisimple Hopf algebras.
\end{abstract}

\maketitle

\section*{Introduction}

Quantum groups were introduced by Drinfeld \cite{dri} and Jimbo \cite{jim} in order to study ``non-classical'' symmetries of complex systems.  This was followed by the fundamental work of Woronowicz \cite{wo1}, \cite{wo2} on compact quantum groups.  The key examples which were constructed by Drinfeld and Jimbo, and further analyzed by Woronowicz, were $q$-deformations $G_q$ of classical Lie groups $G$.  The idea is as follows: consider the commutative algebra $A = C(G)$.  For a suitable choice of generating ``coordinates'' of this algebra, replace commutativity by the $q$-commutation relations $ab = qba$, where $q > 0$ is a parameter.  In this way one obtains an algebra $A_q = C(G_q)$, where $G_q$ is a quantum group.  When $q =1$ one then recovers the classical group $G$.

For $G=O_n,U_n,S_n$ it was later discovered by Wang \cite{wa1}, \cite{wa2} that one can also obtain compact quantum groups by ``removing'' the commutation relations entirely.  In this way one obtains ``free'' versions $O_n^+,U_n^+,S_n^+$ of these classical groups.  This construction has been axiomatized in \cite{bsp} in terms of the ``easiness'' condition for compact quantum groups, and has led to several applications in probability. See \cite{ez2}, \cite{ez3}.

It is clear from the construction that one has $G \subset G^+$ for $G = O_n,U_n,S_n$.  Since $G^+$ can be viewed a ``liberation'' of $G$, it is natural to wonder whether there are any intermediate quantum groups $G \subset G' \subset G^+$, which could be seen as ``partial liberations'' of $G$.  For $O_n,S_n$ this problem has been solved in the case of ``easy'' intermediate quantum groups \cite{bve}, \cite{ez1}.   For $S_n$ there are no intermediate easy quantum groups $S_n \subset G' \subset S_n^+$.  However for $O_n$ there is exactly one intermediate easy quantum group $O_n \subset O_n^* \subset O_n^+$, called the ``half-liberated'' orthogonal group, which was constructed in \cite{bsp}.  At the level of relations among coordinates, this is constructed by replacing the commutation relations $ab = ba$ with the half-commutation relations $abc = cba$.

In the larger category of compact quantum groups it is an open problem whether there are intermediate quantum groups $S_n \subset G \subset S_n^+$, or $O_n \subset G \subset O_n^+$ with $G \neq O_n^*$.  This is an important question for better understanding the ``liberation'' procedure of \cite{bsp}. At $n = 4$ (the smallest value at which $S_n \neq S_n^+$), it follows from the results in \cite{bbi} that the inclusion $S_n \subset S_n^+$ is indeed maximal, and it was conjectured in \cite{bbc} that this is the case, for any $n\in\mathbb N$.  Likewise the inclusion $O_n \subset O_n^* \subset O_n^+$ is known to be maximal at $n = 2$, thanks to the results of Podle\'s in \cite{pod}.  In general it is likely that these two problems are related to each other via combinatorial invariants \cite{bve} or cocycle twists \cite{tga}.

In this paper we make some progress towards solving this problem in the orthogonal case, by showing that the inclusion $O_n\subset O_n^*$ is maximal.  A key tool in our analysis will be the fact the ``projective version'' of $O_n^*$ is the same as that of the classical unitary group $U_n$.  By using a version of  the five lemma for cosemisimple Hopf algebras (following ideas from \cite{abi}, \cite{aga}), we are thus able to reduce the problem to showing that the inclusion of groups $PO_n \subset PU_n$ is maximal.  We then solve this problem by using some Lie algebra techniques inspired from \cite{afg}, \cite{dyn}.

The paper is organized as follows: Section 1 contains background and preliminaries.  In Section 2 we prove that $PO_n\subset PU_n$ is maximal.  In Section 3 we prove a short five lemma for cosemisimple Hopf algebras, which may be of independent interest.  We then use this in Section 4 to prove our main result, namely that $O_n\subset O_n^*$ is maximal.   

\medskip\noindent\textbf{Acknowledgements.}  Part of this work was completed during the Spring 2011 program ``Bialgebras in free probability'' at the Erwin Schr\"odinger Institute in Vienna, and T.B., B.C., S.C. are grateful to the organizers for the invitation.  The work of T.B., J.B., B.C. was supported by the ANR grants ``Galoisint'' and ``Granma'', the work of B.C. was supported by an NSERC Discovery grant and an ERA grant, and the work of S.C. was supported by an NSF postdoctoral fellowship and by the NSF grant DMS-0900776.

\section{Orthogonal quantum groups}

In this section we briefly recall the free and half-liberated orthogonal quantum groups from \cite{wa1}, \cite{bsp}, and the notion of ``projective version'' for a unitary compact quantum group. We will work at the level of Hopf $*$-algebras of representative functions.

First we have the following fundamental definition, arising from Woronowicz' work \cite{wo1}.

\begin{definition}
A \textit{unitary Hopf algebra} is a $*$-algebra $A$ which is generated by elements $\{u_{ij}|1 \leq i,j \leq n\}$ such that $u = (u_{ij})$ and $\overline{u}=(u_{ij}^*)$ are unitaries, and such that:
\begin{enumerate}
 \item There is a $*$-algebra map $\Delta:A \to A \otimes A$ such that $\Delta(u_{ij}) = \sum_{k=1}^n u_{ik} \otimes u_{kj}$.

\item There is a $*$-algebra map $\varepsilon:A \to \mathbb C$ such that $\varepsilon(u_{ij}) = \delta_{ij}$.

\item There is a $*$-algebra map $S:A \to A^{op}$ such that $S(u_{ij}) = u_{ji}^*$.
\end{enumerate}
If $u_{ij} = u_{ij}^*$ for $1 \leq i,j \leq n$, we say that $A$ is an \textit{orthogonal Hopf algebra}.
\end{definition}

It follows that $\Delta,\varepsilon,S$ satisfy the usual Hopf algebra axioms.  The motivating examples of unitary (resp. orthogonal) Hopf algebra is $A = \mathcal R(G)$, the algebra of representative function of a compact subgroup $G \subset U_n$ (resp. $G \subset O_n$).  Here the standard generators $u_{ij}$ are the coordinate functions which take a matrix to its $(i,j)$-entry.

In fact every commutative unitary Hopf algebra is of the form $\mathcal R  (G)$ for some compact group $G \subset U_n$.  In general we use the suggestive notation ``$A = \mathcal R(G)$'' for any unitary (resp. orthogonal) Hopf algebra, where $G$ is a \textit{unitary (resp. orthogonal) compact quantum group}.  Of course any group-theoretic statements about $G$ must be interpreted in terms of the Hopf algebra $A$.  

It can be shown that shown that a unitary Hopf algebra has an enveloping $C^*$-algebra, satisfying Woronowicz' axioms in \cite{wo1}. In general there are several ways to complete a unitary Hopf algebra into a $C^*$-algebra, but in this paper we will ignore this problem and work at the level of unitary Hopf algebras.

The following examples of Wang \cite{wa1} are fundamental to our considerations.

\begin{definition}
The universal unitary Hopf algebra $A_u(n)$ is the universal $*$-algebra generated by  elements $\{u_{ij}|1 \leq i,j \leq n\}$ such that the matrices
$u = (u_{ij})$ and $\overline{u}=(u_{ij}^*)$ in $M_n(A_u(n))$ are unitaries.

The universal orthogonal Hopf algebra $A_o(n)$ is the universal $*$-algebra generated by self-adjoint elements $\{u_{ij}|1 \leq i,j \leq n\}$ such that the matrix $u = (u_{ij})_{1 \leq i,j \leq n}$ in $M_n(A_o(n))$ is orthogonal.
\end{definition}

The existence of the Hopf algebra structural morphisms follows from the universal properties of $A_u(n)$ and $A_o(n)$.  As discussed above, we use the notations $A_u(n)= \mathcal R(U_n^+)$ and 
$A_o(n) = \mathcal R(O_n^+)$, where $U_n^+$ is the \textit{free unitary quantum group} and $O_n^+$ is the \textit{free orthogonal quantum group}.

Note that we have $\mathcal R(O_n^+) \twoheadrightarrow \mathcal R(O_n)$, in fact $\mathcal R(O_n)$ is the quotient of $\mathcal R(O_n^+)$ by the relations that the coordinates $u_{ij}$ commute.  At the level of quantum groups, this means that we have an inclusion $O_n \subset O_n^+$.  

In other words, $\mathcal R(O_n^+)$ is obtained from $\mathcal R(O_n)$ by ``removing commutativity'' among the coordinates $u_{ij}$.  It was discovered in \cite{bsp} that one can obtain a natural orthogonal quantum group by requiring instead that the coordinates ``half-commute''.

\begin{definition}
The half-liberated othogonal Hopf algebra $A_o^*(n)$ is the universal $*$-algebra generated by self-adjoint elements $\{u_{ij}|1 \leq i,j \leq n\}$ which half-commute in the sense that $abc = cba$ for any $a,b,c \in \{u_{ij}\}$, and such that the matrix $u = (u_{ij})_{1 \leq i,j \leq n}$ in $M_n(A_o^*(n))$ is orthogonal.
\end{definition}

The existence of the Hopf algebra structural morphisms again follows from the universal properties of $A_o^*(n)$.  We use the notation $A_o^*(n) = \mathcal R(O_n^*)$, where $O_n^*$ is the \textit{half-liberated orthogonal quantum group}.  Note that we have $\mathcal R(O_n^+) \twoheadrightarrow \mathcal R(O_n^*)  \twoheadrightarrow \mathcal R(O_n)$, i.e. $O_n \subset O_n^* \subset O_n^+$.  As discussed in the introduction, our aim in this paper is to show that the inclusion $O_n \subset O_n^*$ is maximal.  A key tool in our analysis will be the projective version of a unitary quantum group, which we now recall.

\begin{definition}
The projective version of a unitary compact quantum group $G \subset U_n^+$ is the quantum group $PG\subset U_{n^2}^+$, having as basic coordinates the elements $v_{ij,kl}=u_{ik}u_{jl}^*$.
\end{definition}

In other words, $P\mathcal R(G)=\mathcal R (PG)\subset \mathcal R(G)$ is the subalgebra generated by the elements $v_{ij,kl}=u_{ik}u_{jl}^*$. It is clearly a Hopf $*$-subalgebra of $\mathcal R(G)$. In the case where $G \subset U_n$ is classical we recover of course the well-known formula $PG=G/(G\cap \mathbb T)$, where $\mathbb T\subset U_n$ is the group of norm one multiples of the identity.  

The following key result was proved in \cite{bve}.

\begin{theorem}
We have an isomorphism $PO_n^*\simeq PU_n$.
\end{theorem}

\begin{proof}
First, thanks to the half-commutation relations between the standard coordinates on $O_n^*$, for any $a,b,c,d\in\{u_{ij}\}$ we have $abcd=cbad=cdab$. Thus the standard coordinates on the quantum group $PO_n^*$ commute ($ab\cdot cd=cd\cdot ab$), so this quantum group is actually a classical group. A representation theoretic study, based on the diagrammatic results in \cite{bsp}, allows then to show this classical group is actually $PU_n$. See \cite{bve}.
\end{proof}

Note that in fact the techniques developed in the present paper enable us to give a very simple proof of this theorem, avoiding the diagramatic techniques from \cite{bsp}, \cite{bve}. See the last remark in Section 4.

\section{Classical group results}

In this section we prove that the inclusion $PO_n\subset PU_n$ is maximal in the category of compact groups (we assume throughout the paper that $n\geq 2$, otherwise there is nothing to prove). We will see later on, in Sections 3 and 4 below, that this result can be ``twisted'', in order to reach to the maximality of the inclusion $O_n\subset O_n^*$.

Let $\tilde O_{n}$ be the group generated by $O_{n}$ and $\mathbb T\cdot I_{n}$  (the group of multiples of identity of norm one). That is, $\tilde O_n$ is the preimage of $PO_n$ under the quotient map $U_n \twoheadrightarrow PU_n$.  Let $\widetilde{SO}_n \subset \tilde O_n$ be the group generated by $SO_n$ and $\mathbb T\cdot I_n$.  Note that $\tilde{O}_n = \widetilde{SO}_n$ if $n$ is odd, and if $n$ is even then $\tilde{O}_n$ has two connected components and $\widetilde{SO}_n$ is the component containing the identity.

It is a classical fact that a compact matrix group is a Lie group, so $\widetilde{SO}_{n}$ is a Lie group. Let $\mathfrak{so}_{n}$ (resp. $\mathfrak{u}_{n}$)  be the real Lie algebras of $SO_{n}$ (resp. $U_{n}$).  It is known that $\mathfrak{u}_n$ consists of the matrices $M\in M_n(\mathbb C)$  satisfying $M^*=-M$, and $\mathfrak{so}_n=\mathfrak{u}_n\cap M_n(\mathbb R)$. It is easy to see that the Lie algebra of $\widetilde{SO}_{n}$ is $\mathfrak{so}_{n}\oplus i\mathbb{R}$.

First we need the following lemma:

\begin{lemma}\label{soadj}
If $n \geq 2$, the adjoint representation of $SO_n$ on the space of real symmetric matrices of trace zero is irreducible.
\end{lemma}

\begin{proof}
Let $X \in M_n(\mathbb R)$ be symmetric with trace zero, and let $V$ be the span of $\{UXU^t:U \in SO_n\}$.  We must show that $V$ is the space of all real symmetric matrices of trace zero.  

First we claim that $V$ contains all diagonal matrices of trace zero.  Indeed, since we may diagonalize $X$ by conjugating with an element of $SO_n$, $V$ contains some non-zero diagonal matrix of trace zero.  Now if $D = diag(d_1,d_2,\dotsc,d_n)$ is a diagonal matrix in $V$, then by conjugating $D$ by
\begin{equation*}
\begin{pmatrix}
     0 & -1 & 0\\
1 & 0 & 0\\
0 & 0 & I_{n-2}
    \end{pmatrix} \in SO_n
\end{equation*}
we have that $V$ also contains $diag(d_2,d_1,d_3,\dotsc,d_n)$.  By a similar argument we see that for any $1 \leq i,j \leq n$ the diagonal matrix obtained from $D$ by interchanging $d_i$ and $d_{j}$ lies in $V$.  Since $S_n$ is generated by transpositions, it follows that $V$ contains any diagonal matrix obtained by permuting the entries of $D$.  But it is well-known that this representation of $S_n$ on diagonal matrices of trace zero is irreducible, and hence $V$ contains all such diagonal matrices as claimed.

Now if $Y$ is any real symmetric matrix of trace zero, we can find a $U$ in $SO_n$ such that $UYU^{t}$ is a diagonal matrix of trace zero.  But we then have $UYU^t \in V$, and hence also $Y \in V$ as desired.
\end{proof}

\begin{proposition}\label{prop-max-connect}
The inclusion $\widetilde{SO_{n}}\subset U_{n}$ is maximal in the category of connected compact groups. 
\end{proposition}

\begin{proof}
Let $G$ be a connected compact group satisfying $\widetilde{SO}_n \subset G \subset U_n$.  Then $G$ is a Lie group, let $\mathfrak{g}$ denote its Lie algebra, which satisfies $\mathfrak{so}_n\oplus i\mathbb{R} \subset\mathfrak{g}\subset \mathfrak{u}_n$.

Let $ad_{G}$ be the action of $G$ on $\mathfrak{g}$ obtained by differentiating the adjoint action of $G$ on itself. 
This action turns $\mathfrak{g}$ into a $G$-module. Since $SO_n \subset G$, $\mathfrak{g}$ is also an $SO_n$-module.

Now if $G \neq \widetilde{SO}_n$, then since $G$ is connected we must have $\mathfrak{so}_n\oplus i\mathbb{R}\neq\mathfrak{g}$. 
It follows from the real vector space structure of the Lie algebras $\mathfrak{u}_n$ and $\mathfrak{so}_{n}$ that
there exists a non-zero symmetric real matrix of trace zero $X$ such that $iX\in \mathfrak{g}$.

But by Lemma \ref{soadj} the space of symmetric real matrices of trace zero is an irreducible representation of $SO_{n}$ under the adjoint action.  So $\mathfrak{g}$ must contain all such $X$, and hence $\mathfrak{g}=\mathfrak{u}_{n}$.  But since $U_n$ is connected, it follows that $G = U_n$.  
\end{proof}

Our aim is to extend this result to the category of compact groups.  To do this we need to compute the \textit{normalizer} of $\widetilde{SO}_n$ in $U_n$, i.e. the subgroup of $U_n$ consisting of unitary $U$ for which $U^{-1}XU \in \widetilde{SO}_n$ for all $X \in \widetilde{SO}_n$.  For this we need two lemmas.

\begin{lemma}\label{commutant}
 The commutant of $SO_n$ in $ M_n(\mathbb C)$, denoted $SO_n'$, is as follows:
 \begin{enumerate}
  \item $SO_2' = \{\begin{pmatrix}
                    \alpha & \beta \\
		    -\beta & \alpha
                   \end{pmatrix}, \ \alpha , \beta \in \mathbb C\}.$
\item If $n \geq 3$, $SO_n' = \{\alpha I_n, \alpha \in \mathbb C\}.$
 \end{enumerate}  
\end{lemma}

\begin{proof}
 At $n=2$ this is a direct computation. For $n \geq 3$, an element in $X  \in SO_n'$ commutes with any diagonal matrix having exactly $n-2$  entries equal to $1$ and  two entries equal to $-1$. Hence
 $X$ is a diagonal matrix. Now since $X$ commutes with any even permutation matrix and $n \geq 3$,
 it commutes in particular with the permutation matrix associated with the cycle $(i,j,k)$ for any $1<i<j<k$, and hence  all 
 the entries of $X$ are the same: we conclude that $X$ is a scalar matrix. 
\end{proof}

\begin{lemma}\label{dense}
The set of matrices with  non-zero trace is dense in $SO_n$.  
\end{lemma}

\begin{proof}
 At $n=2$ this is clear since the set of elements in $SO_2$ having a given trace is finite.
 Assume that $n>2$ and let $T \in SO_n \simeq SO(\mathbb R^n)$ with $Tr(T)=0$. 
 Let $E\subset \mathbb R^n$ be a 2-dimensional subspace preserved by $T$ and such that
 $T_{|E} \in SO(E)$. Let $\epsilon >0$ and let $S_\epsilon \in SO(E)$ with
 $||T_{|E}-S_\epsilon|| <\epsilon$ and $Tr(T_{|E}) \not= Tr(S_\epsilon)$ ($n=2$ case).
 Now define $T_\epsilon \in SO(\mathbb R^n)=SO_n$ by $T_{\varepsilon | E} = S_\epsilon$ and
 $T_{\epsilon|E^\perp}=T_{|E^\perp}$. It is clear that 
 $||T-T_\epsilon|| \leq ||T_{|E}-S_\epsilon||<\epsilon$ and that
 $Tr(T_\epsilon)=Tr(S_\epsilon)+Tr(T_{|E^\perp}) \not = 0$. 
\end{proof}

\begin{proposition}\label{prop-normalize}
$\tilde O_n$ is the normalizer of $\widetilde{SO}_n$ in $U_n$.
\end{proposition}

\begin{proof}
It is clear that $\tilde O_n$ normalizes $\widetilde{SO}_n$, so we must show that if $U \in U_n$ normalizes $\widetilde{SO}_n$ then $U \in \tilde O_n$.  First note that $U$ normalizes $SO_n$. Indeed if $X \in SO_n$ then $U^{-1}XU \in \widetilde {SO}_n$, so 
$U^{-1}XU= \lambda Y$ for $\lambda \in \mathbb T$ and $Y \in SO_n$. If $Tr(X) \not=0$, we have $\lambda \in \mathbb R$ and hence $\lambda Y = U^{-1}XU \in SO_n$.
The set of matrices having non-zero trace is dense in $SO_n$ by Lemma \ref{dense}, so since $SO_n$ is closed
and the matrix operations are continous, we conclude that $U^{-1}XU \in SO_n$ for all $X \in SO_n$.

Thus for any $X \in SO_n$, we have
$(UXU^{-1})^t(UXU^{-1})= I_n$ and hence $X^tU^tUX= U^tU$. This means that $U^tU \in SO_n'$. 
Hence if $n\geq 3$, we have $U^tU=  \alpha I_n$ by Lemma \ref{commutant}, with $\alpha \in \mathbb T$ since $U$ is unitary. Hence we have $U = \alpha^{1/2} (\alpha^{-1/2} U)$ with $\alpha^{-1/2} U \in O_n$, and $U \in \widetilde{O_n}$. If $n=2$, Lemma \ref{commutant} combined with the fact that 
$(U^tU)^t=U^tU$ gives again that $U^tU=  \alpha I_2$, and we conclude as in the previous case.
\end{proof}

We can now extend Proposition \ref{prop-max-connect} as follows.

\begin{proposition}\label{prop-class-max}
The inclusion $\tilde O_n \subset U_n$ is maximal in the category of compact groups.
\end{proposition}

\begin{proof}
Suppose that $\tilde O_n \subset G \subset U_n$ is a compact group such that $G \neq U_n$.  It is a well known fact that the connected component of the identity in $G$ is a normal subgroup, denoted $G_0$.  Since we have $\widetilde{SO}_n \subset G_0 \subset U_n$, by Proposition \ref{prop-max-connect} we must have $G_0 = \widetilde{SO}_n$.  But since $G_0$ is normal in $G$, $G$ normalizes $\widetilde{SO}_n$ and hence $G \subset \tilde O_n$ by Proposition \ref{prop-normalize}.
\end{proof}

We are now ready to state and prove the main result in this section.

\begin{theorem}\label{thm-proj-max}
The inclusion $PO_n\subset PU_n$ is maximal in the category of compact groups.
\end{theorem}

\begin{proof}
It follows directly from the observation that the maximality of $\tilde O_{n}$ in $U_{n}$ 
implies the maximality of $PO_n$ in  $PU_n$.  Indeed, if $PO_n \subset G \subset PU_n$ were an intermediate subgroup, then its preimage under the quotient map $U_n \twoheadrightarrow PU_n$ would be an intermediate subgroup of $\tilde O_{n}\subset U_{n}$, contradicting Proposition \ref{prop-class-max}.
\end{proof}

\section{A short five lemma}

In this section we prove a short five lemma for cosemisimple Hopf algebras (Theorem 3.4 below), which is a result having its own interest, to be used in Section 4 below. 

\begin{definition}
A sequence of Hopf algebra maps
$$\mathbb C \to B\overset{i}\to A\overset{p}\to L\to\mathbb C$$
is called pre-exact if $i$ is injective, $p$ is surjective and $i(A)=H^{cop}$, where:
$$A^{cop}=\{a\in A|(id\otimes p)\Delta(a)=a\otimes 1\}$$
\end{definition}

The example that we are interested in is as follows.

\begin{proposition}
Let $A$ be an orthogonal Hopf algebra with  generators $u_{ij}$. Assume that we have surjective Hopf algebra map $p:A\to\mathbb C\mathbb Z_2$, $u_{ij}\to\delta_{ij}g$, where $<g>=\mathbb Z_2$. Let $PA$ be the projective version of $A$, i.e. the subalgebra generated by the elements $u_{ij}u_{kl}$ with the inclusion $i:PA \subset A$. Then the sequence
$$\mathbb C\to PA\overset{i'}\to A\overset{p}\to\mathbb C\mathbb Z_2 \to
\mathbb C$$
is pre-exact.
\end{proposition}

\begin{proof}
We have:
$$(id \otimes p)\Delta(u_{i_1j_1}\ldots u_{i_mj_m})= 
\begin{cases}
u_{i_1j_1}\ldots u_{i_mj_m}\otimes 1&\text{if $m$ is even} \\
u_{i_1j_1}\ldots u_{i_mj_m}\otimes g&\text{if $m$ is odd}
\end{cases}$$ 
Thus $H^{cop}$ is the span of monomials of even length, which is clearly $PH$.
\end{proof}

A pre-exact sequence as in Definition 3.1 is said to be exact \cite{ade} if in addition we have $i(A)^+H=\ker(\pi)=Hi(A)^+$, where $i(A)^+=i(A)\cap\ker(\varepsilon)$. The pre-exact sequence in Proposition 3.2 is actually exact, but we only need its pre-exactness in what follows.

In order to prove the short five lemma, we use the following well-known result. We give a proof for the sake of completness.

\begin{lemma}\label{integral-injective}
Let $\theta:A\to A'$ be a Hopf algebra morphism with $A,A'$ cosemisimple and let
 $h_A,h_{A'}$ be the respective Haar integrals of $A,A'$. Then $\theta$ is injective iff
 $h_{A'}\theta=h_A$.
\end{lemma}

\begin{proof}
For $a \in A$, we have:
$$\theta (h_{A'}(\theta(a_1))a_2)=h_{A'}(\theta(a)_1)\theta(a)_2=\theta(h_{A'}\theta(a)1)$$

Thus if $\theta$ is injective then $h_{A'}\theta$ is a Haar integral on $A$, and the result follows from the uniqueness of the Haar integral.
 
Conversely, assume that $h_A=h_{A'}\theta$. Then for all $a,b \in A$, we have $h_A(xy)=h_{A'}(\theta(a)\theta(b))$, so if $\theta(a)=0$, we have $h_A(ab)=0$ for all $b\in H$. It follows from the orthogonality relations that $a=0$, and hence $\theta$ is injective.
\end{proof}

\begin{theorem}\label{short}
Consider a commutative diagram of cosemisimple Hopf algebras
$$\begin{CD}
k@>>>B@>{i}>>A@>{\pi}>>L@>>>k\\
@.@|@VV{\theta}V@|@.\\
k@>>>B@>{i'}>>A'@>{\pi'}>>L@>>>k
\end{CD}$$
where the rows are pre-exact. Then $\theta$ is injective.
\end{theorem}

\begin{proof}
We have to show that $h_A=h_{A'}\theta$, where $h_A,h_{A'}$ are the respective Haar integrals of $A,A'$. Let $\Lambda$ be the set of isomorphism classes of simple $L$-comodules and consider the Peter-Weyl decomposition of $L$:
$$L=\bigoplus_{\lambda\in\Lambda}L({\lambda})$$

We view $A$ as a right $L$-comodule via $(id\otimes\pi)\Delta$. Then $A$ has a decomposition into isotypic components as follows, where $A_\lambda =\{a \in A \ | \ (id \otimes \pi)\circ \Delta(a) \in A \otimes L(\lambda)\}$: 
$$A=\bigoplus_{\lambda\in\Lambda}A_\lambda$$ 

It is clear that $A_1=A^{co\pi}$. Then if $\lambda\neq 1$, we have $h_A(A_\lambda)=0$. Indeed for $a\in A_\lambda$,
we have: 
$$a_1\otimes\pi(a_2)\in H\otimes L(\lambda)
\implies h_A(a)1=\pi(h_H(a_1)a_2)\in L(\lambda)
\implies h_H(a)=0$$ 

Since $\pi' \theta=\pi$, it is easy to see that $\theta(A_\lambda)\subset A'_\lambda$
and hence for $\lambda\neq 1$, $h_{A'|A'_\lambda}=h_{A'}\theta_{|A_\lambda}=0=h_{A| A_\lambda}$. For $\lambda=1$, we have $i(A)=A_1$ and $\theta$ is injective on $i(A)$ since $\theta i=i'$. Hence by  Lemma \ref{integral-injective} we have $h_{A'}\theta_{|A_1}=h_{A_1}= h_{A|A_1}$. Since $A=\oplus_{\lambda\in\Lambda}A_\lambda$ we conclude
$h_A=h_{A'}\theta$ and by Lemma \ref{integral-injective} we get that $\theta$ is injective.
\end{proof}

\section{The main result}

We have now all the ingredients for stating and proving our main result in this paper.

\begin{theorem}\label{main}
The inclusion $O_n \subset O_n^*$ is maximal in the category of compact quantum groups.
\end{theorem}

\begin{proof}
Consider a sequence of surjective Hopf $*$-algebra maps as follows, whose  composition is the canonical surjection:
$$A_o^*(n)\overset{f}\to A\overset{g}\to\mathcal R(O_n)$$

By Proposition 3.2 we get a commutative diagram of Hopf algebra maps with pre-exact rows:
$$\begin{CD}
\mathbb C@>>>PA_o^*(n)@>{i_1}>>A_o^*(n)@>{p_1}>>\mathbb C\mathbb Z_2 @>>>\mathbb C\\
@.@VV{f_|}V@VV{f}V@|@.\\
\mathbb C@>>>PA@>{i_2}>>A@>{p_2}>>\mathbb C\mathbb Z_2@>>>\mathbb C\\
@.@VV{g_|}V@VV{g}V@|@.\\
\mathbb C@>>>P\mathcal R(O_n)@>{i_3}>>\mathcal R(O_n)@>{p_3}>>\mathbb C\mathbb Z_2 @>>>\mathbb C 
\end{CD}$$
 
Consider now the following composition, with the isomorphism on the left coming from Theorem 1.5: 
$$\mathcal R(PU_n)\simeq PA_o^*(n)\overset{f_|} \to PA\overset{g_|}\to P\mathcal R(O_n)\simeq\mathcal R(PO_n)$$

This induces, at the group level, the embedding $PO_n\subset PU_n$. By Theorem \ref{thm-proj-max} $f_|$ or $g_|$ is an isomorphism. If $f_|$ is an isomorphism we get a commutative diagram of Hopf algebra morphisms with pre-exact rows:
$$\begin{CD}
\mathbb C@>>>PA_o^*(n)@>{i_1}>>A_o^*(n)@>{p_1}>>\mathbb C\mathbb Z_2
@>>>\mathbb C\\
@.@|@VV{f}V@|@.\\
\mathbb C@>>>PA_o^*(n)@>{i_2 \circ f_|}>>A@>{p_2}>>\mathbb C\mathbb Z_2  @>>>\mathbb C\\
\end{CD}$$
 
Then $f$ is an isomorphism by Theorem 3.4. Similarly if $g_|$ is an isomorphism, then $g$ is an isomorphism.
\end{proof}

Observe that the technique in the proof of Theorem \ref{main} also enables us to prove that $PO_n^* \simeq PU_n$ independently from \cite{bve}. Indeed, since $PA_o^*(n)$ is commutative, there exists a compact group $G$ with $PA_o^*(n) \simeq \mathcal R(G)$ and $PO_n \subset G \subset PU_n$. Then  Theorem \ref{thm-proj-max} gives $G=PO_n$ or $G=PU_n$. If $G=PO_n$, then as in the proof of Theorem \ref{main}, Theorem \ref{short} gives that $A_o^*(n) \twoheadrightarrow \mathcal R(O_n)$ is an isomorphism, which is false since $A_o^*(n)$ is a not commutative if $n\geq 2$. Hence $G= PU_n$.


\begin{thebibliography}{99}

\bibitem{abi}N. Andruskiewitsch and J. Bichon, Examples of inner linear Hopf algebras, {\em Rev. Un. Mat. Argentina} {\bf 51} (2010), 7--18. 

\bibitem{ade}N. Andruskiewitsch and J. Devoto, Extensions of Hopf algebras, {\em St. Petersburg Math. J.}  {\bf 7} (1996), 17--52.

\bibitem{aga}N. Andruskiewitsch and G.A. Garcia, Quantum subgroups of a simple quantum group at roots of $1$, {\em Compos. Math.} {\bf 145} (2009), 476--500.
 
\bibitem{afg}F. Antoneli, M. Forger and P. Gaviria, Maximal subgroups of compact Lie groups, \href{http://arxiv.org/abs/math/0605784}{\tt arxiv:0605784}.

\bibitem{bbi}T. Banica and J. Bichon, Quantum groups acting on 4 points, {\em J. Reine Angew. Math.} {\bf 626} (2009), 74--114.

\bibitem{bbc}T. Banica, J. Bichon and B. Collins, Quantum permutation groups: a survey, {\em Banach Center Publ.} {\bf 78} (2007), 13--34.

\bibitem{tga}T. Banica, J. Bichon and S. Curran, Quantum automorphisms of twisted group algebras and free hypergeometric laws, {\em Proc. Amer. Math. Soc.}, to appear.

\bibitem{ez1}T. Banica, S. Curran and R. Speicher, Classification results for easy quantum groups, {\em Pacific J. Math.} {\bf 247} (2010), 1--26.

\bibitem{ez2}T. Banica, S. Curran and R. Speicher, Stochastic aspects of easy quantum groups, {\em Probab. Theory Related Fields} {\bf 149} (2011), 435--462.

\bibitem{ez3}T. Banica, S. Curran and R. Speicher, De Finetti theorems for easy
  quantum groups, {\em Ann. Probab.}, to appear.

\bibitem{bsp}T. Banica and R. Speicher, Liberation of orthogonal Lie groups, {\em Adv. Math.} {\bf 222} (2009), 1461--1501. 

\bibitem{bve}T. Banica and R. Vergnioux, Invariants of the half-liberated orthogonal group, {\em Ann. Inst. Fourier} {\bf 60} (2010), 2137--2164.

\bibitem{dri}V. Drinfeld, Quantum groups, Proc. ICM Berkeley (1986), 798--820.

\bibitem{dyn}E.B. Dynkin, Maximal subgroups of the classical groups, {\em AMS Transl.} {\bf 6} (1957), 245--378.

\bibitem{jim}M. Jimbo, A $q$-difference analogue of $U({\mathfrak g})$ and the
  Yang-Baxter equation, {\em Lett. Math. Phys.} {\bf 10} (1985), 63--69.

\bibitem{pod}P. Podle\'s, Symmetries of quantum spaces. Subgroups and quotient spaces of quantum SU(2) and SO(3) groups, {\em Comm. Math. Phys.} {\bf 170} (1995), 1--20.

\bibitem{wa1}S. Wang, Free products of compact quantum groups, {\em Comm. Math. Phys.} {\bf 167} (1995), 671--692.

\bibitem{wa2}S. Wang, Quantum symmetry groups of finite spaces, {\em Comm. Math. Phys.} {\bf 195} (1998), 195--211.

\bibitem{wo1}S.L. Woronowicz, Compact matrix pseudogroups, {\em Comm. Math. Phys.} {\bf 111} (1987), 613--665.

\bibitem{wo2}S.L. Woronowicz, Tannaka-Krein duality for compact matrix pseudogroups. Twisted SU(N) groups, {\em Invent. Math.} {\bf 93} (1988), 35--76.
\end{thebibliography}
\end{document}